\documentclass[12pt]{article}
\usepackage[T1]{fontenc}
\usepackage{amssymb, amsmath, amsthm}
\usepackage[all]{xy}
\usepackage{tikz-cd}
\usepackage{newtxtext,stmaryrd}
\usepackage{hyperref}
\usepackage{accents}
\usepackage{hhline}
\usepackage{multicol}

\usetikzlibrary{backgrounds,matrix}

\DeclareFontEncoding{OT2}{}{}

\newcommand{\mr}[1]{\mathrm{#1}}
\newcommand{\mf}[1]{\mathfrak{#1}}
\newcommand{\mc}[1]{\mathcal{#1}}
\newcommand{\mb}[1]{\mathbb{#1}}

\newcommand{\ms}[1]{\mathsf{#1}}

\newcommand{\Z}{\mb{Z}}
\newcommand{\Q}{\mb{Q}}
\newcommand{\zp}{\mb{Z}_p}
\newcommand{\qp}{\mb{Q}_p}

\newcommand{\La}{\Lambda}
\newcommand{\F}{\mb{F}}
\newcommand{\R}{\mb{R}}
\newcommand{\C}{\mb{C}}

\newcommand{\qbar}{\overline{\Q}}

\newcommand{\smatrix}[1]{\left(\begin{smallmatrix}#1\end{smallmatrix}\right)}

\newcommand{\ps}[1]{\llbracket #1 \rrbracket}

\DeclareMathOperator{\Hom}{Hom} \DeclareMathOperator{\Aut}{Aut}
\DeclareMathOperator{\End}{End} \DeclareMathOperator{\Gal}{Gal}
 \DeclareMathOperator{\res}{res}

\DeclareMathOperator{\Spec}{Spec}

\DeclareMathOperator{\St}{St}

\DeclareMathOperator{\Res}{Res}
\DeclareMathOperator{\dlog}{dlog}
\DeclareMathOperator{\cond}{cond}

\newcommand{\et}{\mr{\acute{e}t}}
\newcommand{\id}{\mr{id}}
\newcommand{\SL}{\mathrm{SL}}
\newcommand{\GL}{\mathrm{GL}}

\newcommand{\ord}{\mathrm{ord}}
\newcommand{\gm}{\mb{G}_m}

\newcommand{\quo}{\mr{quo}}
\newcommand{\sub}{\mr{sub}}
\newcommand{\Fr}{\mr{Fr}}

\newtheorem{theorem}{Theorem}

\newtheorem{step}{Step}

\newtheorem{conjecture}[theorem]{Conjecture}

\theoremstyle{definition}

\theoremstyle{remark}
\newtheorem{remark}[theorem]{Remark}
\newtheorem*{ack}{Acknowledgments}

\begin{document}

\title{Through the lens of the Eisenstein ideal}
\author{Romyar Sharifi}
\date{}
\maketitle

\begin{abstract}
	This paper exposits relationships between the geometry of modular curves and the arithmetic 
	of cyclotomic fields through the Eisenstein ideal. For an integer $N \ge 5$ and an odd prime $p$, 
	we define two conjecturally inverse maps between the real part $P$ of the Eisenstein reduction
	of $H_1(X_1(N),\zp)$ and the real part $Y$ of $K_2(\Z[\mu_N]) \otimes \zp$.
	We recall a motivic construction of the first map $\varpi \colon P \to Y$ and describe the 
    	recent proof of the Eisenstein property
	of the underlying map on modular symbols. We provide a fully equivariant construction of the second map
	$\Upsilon \colon Y \to P$ using the Eisenstein reduction of the first \'etale cohomology of $X_1(N)$. 
	We then explain an equivariant construction of the Euler system of cyclotomic units using extension classes in
	relative cohomology groups of modular curves and describe its place in our program.
\end{abstract}

The seminal paper \cite{ribet} of Ribet on the converse to Herbrand's theorem demonstrated that congruences between Eisenstein series and cusp forms link the geometry of modular curves with the arithmetic of cyclotomic fields. This link was exploited in the work \cite{mw} of Mazur and Wiles to prove the main conjecture of Iwasawa theory. Their work supplied one divisibility, and the analytic class number formula then implied the other. In more general settings, a different method of proof of the opposite divisibility is needed, and that of Euler systems has served this purpose.

The Ribet-Mazur-Wiles method, as in Ohta's work \cite{ohta-ord2},
involves a reduction of the two-dimensional Galois representation given by an Eisenstein component of an inverse limit of first \'etale cohomology groups of modular curves in a $p$-adic tower. A lattice in such a residually reducible representation is noncanonical, so one chooses an extremal Galois-stable lattice, fixing a $p$-adic local quotient and making the kernel as small as possible. The residual Galois action supplies a surjective homomorphism from an eigenspace of the unramified Iwasawa module $X_{\infty}$ to a Hecke module 
with a $p$-adic $L$-function dividing its characteristic ideal. A comparison of Galois and Hecke actions gives the divisibility.

There is another lattice which stares one in the face, formed from the cohomology of the closed modular curves. In certain eigenspaces, we used this lattice to produce a canonical map $\Upsilon$ from a variant $Y_{\infty}$ of $X_{\infty}$ to the Eisenstein reduction of the real part $\mc{S}_{\infty}$ of the inverse limit of the first homology groups of the closed modular curves \cite{sh-Lfn}. We conjectured that $\Upsilon$ is an isomorphism, which implies that the cohomological lattice is the extremal one.

Busuioc \cite{busuioc} and the author \cite{sh-Lfn} then defined an explicit map from $\mc{S}_{\infty}$ to $Y_{\infty}$. At finite level, it sends Manin symbols in homology to cup products of cyclotomic units in second Galois cohomology. That this map factors through the Eisenstein quotient was proven by Fukaya and Kato \cite{fk-sh} and recently in a stronger conjectured form by Xu \cite{xu}, building on our work with Venkatesh \cite{sv}. I also conjectured in \cite{sh-Lfn} that the resulting map $\varpi$ is inverse to $\Upsilon$ in the allowed eigenspaces, which amounts to an explicit refinement of the main conjecture. In \cite{fk-sh}, Fukaya and Kato proved two strong results towards this conjecture.

The question remained whether Euler systems were amenable to a similar geometric interpretation. Urban \cite{urban} used critical-slope Eisenstein series to produce congruences. The author and Wake gave an approach to the construction of the cyclotomic Euler system in line with the author's program in unpublished work \cite{shwa}. Work outlined in \cite{skinner} and detailed in a recent preprint of Shang, Skinner, and Trejos \cite{sst} gives a closely related approach.

In this article, we exposit the above constructions and their interrelationships, with complete details found in or deferred to other works. 
The highlights are an equivariant construction of $\Upsilon$ that allows us to extend our conjecture, arising from ongoing work with Oh \cite{os}, and an equivariant construction of the cyclotomic Euler system in the cohomology of modular curves, emanating from \cite{shwa}. Though we lack space to discuss generalizations, we present our constructions with these in mind. 

\section*{From geometry to arithmetic}

\subsection*{Some notation}

We fix throughout a positive integer $N \ge 5$ and define numerous objects that depend on $N$ without using
$N$ in the notation. We also fix an odd prime $p$ to work with $\zp$-coefficients. Write $N = Mp^r$ with $p \nmid M$. When we consider inverse limits over increasing $r$, we adorn the limit objects with a subscript $\infty$ and the objects at level $Mp^r$ with a subscript $r$. 

Let $\mf{S}$ and $\mf{M}$ denote the cusp forms and modular forms of weight $2$ and level $\Gamma_1(N)$ with $\zp$-coefficients, respectively.  We have full and cuspidal modular Hecke algebras $\mf{H} \subset \End_{\zp}(\mf{M})$ and $\mf{h} \subset \End_{\zp}(\mf{S})$. Inside $\mf{H}$, we have the Eisenstein ideal $I$ generated by $T_{\ell} - 1 - \ell \langle \ell \rangle$ for primes $\ell \nmid N$ and $U_{\ell} - 1$ for primes $\ell \mid N$. We likewise denote its image in $\mf{h}$ by $I$. Unless otherwise noted, we view $Y_1(N)$ and $X_1(N)$ as schemes over $\Z[\frac{1}{N}]$ with $0 \in X_1(N)(\Q)$; when taking homology, we mean their complex points.

We view $\qbar$ as the algebraic numbers in $\C$ and set $\zeta_N = e^{2\pi i/N}$.
Let $D = (\Z/N\Z)^{\times}/\langle -1 \rangle$ and $\Lambda = \zp[D]$. For $d \in \Z$ prime to $N$, let $[d] \in \La$ denote the group element of $d$. 
 Set $G = \Gal(\Q(\mu_N)^+/\Q)$, and let $\sigma_d \in \Gal(\Q(\mu_N)/\Q)$ satisfy $\sigma_d(\zeta_N) = \zeta_N^d$.  
We shall consider two types of $\Lambda$-modules:
$\mf{H}$-modules on which $[d]$ acts as $\langle d \rangle$, and
$\zp[G]$-modules on which $[d]$ acts as $\sigma_d^{-1}$.\footnote{Papers on the subject have inconsistent conventions regarding $\La$-actions, and we have chosen a new one here.}

Let $\varsigma$ denote complex conjugation.
For a $\zp[\langle \varsigma \rangle]$-module $A$, we let $A^{\pm}$ denote its $(\pm 1)$-eigenspaces for $\varsigma$ and $A_{\pm}$ its maximal quotients on which $\varsigma$ acts as $\pm 1$. We identify these via the quotient maps $A^{\pm} \to A_{\pm}$.

\subsection*{The map $\varpi$}

As a consequence of the Eichler-Shimura isomorphism, the eigenspaces $H_1(X_1(N),\qp)^{\pm}$ of $\pm 1$ for the action of $\varsigma$ are each free of rank $1$ over $\mf{h} \otimes_{\zp} \qp$. 
Let $C'_0$ denote the nonzero cusps on $X_1(N)$, which is to say those that do not lie over $0 \in X_0(N)$. We consider the topological
$\mf{H}$-modules
$$
	\mc{S} = H_1(X_1(N),\zp)_+ \hookrightarrow \mc{S}^{\circ} = H_1(X_1(N),C'_0,\zp)_+.
$$
The group $\mc{S}^{\circ}$ has a simple presentation on Manin symbols $[u:v]$ with $u, v \in \Z/N\Z - \{0\}$ and $(u,v) = 1$. To be precise, here $[u:v]$ is the image of the level $N$ Atkin-Lehner operator $\tau_N \colon z \mapsto \frac{-1}{Nz}$ applied to the class $\{\frac{b}{d} \to \frac{a}{c}\}$ of the geodesic from $\frac{b}{d}$ to $\frac{a}{c}$ in the extended complex upper half-plane for $\smatrix{a&b\\c&d} \in \SL_2(\Z)$ with $(u,v) = (c,d) \bmod N\Z^2$.

On the arithmetic side, we have the $\zp[G]$-modules
$$
	Y = (K_2(\Z[\mu_N]) \otimes \zp)_+ \hookrightarrow Y^{\circ} = (K_2(\Z[\mu_N,\tfrac{1}{N}]) \otimes \zp)_+.
$$
As in \cite{busuioc, sh-Lfn}, there is a homomorphism
$\Pi^{\circ} \colon \mc{S}^{\circ} \to Y^{\circ}$ of $\La$-modules
given by 
$$
	\Pi^{\circ}([u:v]) = \{1-\zeta_N^u,1-\zeta_N^v\},
$$ 
where $\{ \ ,\, \}$ denotes the image of the Steinberg symbol in $Y^{\circ}$.

The following conjecture of the author is now proven due to work of Fukaya-Kato \cite{fk-sh}, the author and Venkatesh \cite{sv}, 
and Xu \cite{xu}.

\begin{theorem} \label{Eis_prop}
	The map $\Pi^{\circ}$ is Eisenstein: i.e., $I\mc{S}^{\circ} \subset \ker \Pi^{\circ}$.
\end{theorem}

Let $\Delta$ be the semigroup of matrices in $\GL_2(\Q)$ with integer entries, and consider the scheme $\gm^2$ over $\Q$ with its right action of $\Delta$. Let $\tilde{\Gamma}_0(N)$ be the subgroup of $\Gamma = \GL_2(\Z)$ of matrices with lower left entry in $N\Z$.
In \cite{sv}, we construct a $\Gamma$-cocycle $\Theta$ from the motivic cohomology of subschemes of $\gm^2$ and use it
to obtain a prime-to-level Eisenstein property of $\Pi^{\circ}|_{\mc{S}}$.

\begin{proof}[Sketch of construction.]
	Consider the Gersten complex $\ms{K} = [\ms{K}_2 \to \ms{K}_1 \to \ms{K}_0]$  
	with terms
	$$
		K_2(\Q(\gm^2)) \xrightarrow{\partial} \bigoplus_S \Q(S)^{\times} \xrightarrow{\partial} \bigoplus_D \Z,
	$$
	where $S$ and $D$ run over irreducible dimension one and zero cycles on $\gm^2$, and
	the residue maps $\partial$ are sums of tame symbols and divisor maps. The semigroup $\Delta$
	acts on $\ms{K}$ by pullback.
	As $\ms{K}$ is right exact and $H_2(\ms{K})$ is the motivic cohomology group $H^2(\gm^2,\Z(2))$, we set
	$\bar{\ms{K}}_2 = \ms{K}_2/H^2(\gm^2,\Z(2))$ to obtain an exact sequence.

	Fix $e = 1$ in the $\{1\}$-summand of $\ms{K}_0$ and $\eta = 1-z_1^{-1}$ in the $\{z_2 = 1\}$-summand
	of $\ms{K}_1$, for $z_i$ the $i$th coordinate function.
	Since $\partial(\eta) = e$, the connecting map $\ms{K}_0^{\Gamma} \to H^1(\Gamma,\bar{\ms{K}}_2)$
	takes $e$ to the class $[\Theta]$ of a cocycle $\Theta \colon \Gamma \to \bar{\ms{K}}_2$ 
	with $\partial \circ \Theta = d\eta$.
	This cocycle has three properties inherited from $e$ and $\eta$:
	\begin{itemize}
		\item it is \emph{Eisenstein} as $T_{\ell}[\Theta] = (\ell + \smatrix{\ell&0\\0&\ell}^*) [\Theta]$ for all primes $\ell$, 
		\item it is \emph{parabolic} as it is trivial on matrices with upper row $(1,0)$, and
		\item its values satisfy an \emph{explicit formula} as sums of $\SL_2(\Z)$-pullbacks of $\{1-z_1,1-z_2\}$.
	\end{itemize}
	
	By analyzing fixed parts under trace maps $\smatrix{n&0\\0&n}_*$ of positive integers $n$, 
	we can replace $\Theta$ by a lift $\Theta \colon \Gamma \to \ms{K}_2/Z$, where 
	$Z = \langle \{-z_1,-z_2\} \rangle$. 
	We have $\ms{K}_2 = \varinjlim_U H^2(U,\Z(2))$ for opens $U$ in $\gm^2$.
	The restriction $\Theta|_{\tilde{\Gamma}_0(N)}$ 
	takes values in $\ms{K}_{2,N}/Z$ for the analogously defined limit $\ms{K}_{2,N}$ over opens containing $(1,\zeta_N)$.
	
	The pullback of $\{-z_1,-z_2\}$ under $(1,\zeta_N)$ is $\{-1,-\zeta_N\}$, which has order dividing $2$.
	Hence, the $p$-part of the pullback $(1,\zeta_N)^* \circ \Theta|_{\tilde{\Gamma}_0(N)}$ is a cocycle 
	$\Theta_N \colon \tilde{\Gamma}_0(N) \to K_2(\Q(\mu_N)) \otimes_{\Z} \zp$,
	where $\smatrix{a&b\\c&d} \in \tilde{\Gamma}_0(N)$ acts as $\sigma_d$.
	This cocycle is parabolic, Eisenstein for 
	$T_{\ell}$ with $\ell \nmid N$, takes values in $K_2(\Z[\mu_N,\frac{1}{N}]) \otimes \zp$, and 
	has an explicit formula as sums of Steinberg symbols $\{1-\zeta_N^u,1-\zeta_N^v\}$. 
	The action of $\Gamma_1(N)$ in $\SL_2(\Z)$ on the coefficients is trivial, so the restriction of $\Theta_N$
	to this group is a homomorphism. Parabolicity and the switch from $\GL_2(\Z)$ to $\SL_2(\Z)$ give that
	the resulting homomorphism factors through $\mc{S}$. Precomposing with $\tau_N$
	yields $\Pi^{\circ}|_{\mc{S}}$ with its explicit formula, and the Eisenstein property of $\Theta_N$ implies 
	that $\Pi^{\circ}|_{\mc{S}}$ factors through the quotient by the action 
	of the operators $T_{\ell} - 1 -\ell \langle \ell \rangle 
	= \langle \ell \rangle(T_{\ell}^* - \ell - \langle \ell \rangle^*)$ for $\ell \nmid N$.
\end{proof}

After \cite{sv}, Lecouturier and Wang \cite{lewa} used level compatibility to show that $\Pi^{\circ}|_{\mc{S}}$ is Eisenstein for all operators
so long as $p \ge 5$. In \cite{xu}, Xu replaces the cocycle with a map from a Steinberg module to prove the full result. These Steinberg modules consist of degree zero divisors on sets of cusps so have the advantage that their $\Gamma_1(N)$-coinvariants are homology groups relative to these cusps.

\begin{proof}[Rough sketch of proof of Theorem \ref{Eis_prop}]
	The Steinberg module $\mr{St}$ for $\GL_2(\Q)$ is the $\Z[\GL_2(\Q)]$-module of degree zero divisors on $\mb{P}^1(\Q)$. 
	A computation of tame symbols shows that there is a $\GL_2(\Q)$-equivariant modular symbol
	$$
		\Phi \colon \St \to \ms{K}_2/W, 
		\quad \Phi([\tfrac{a}{c}]-[\tfrac{b}{d}]) = \smatrix{a&b\\c&d} \{1-z_1,1-z_2\}
	$$ 
	for an explicit subgroup $W$ containing $Z$. Here, we have extended the pushforward-by-adjugate action of $\Delta$
	to $\GL_2(\Q)$ on the image of $\Phi$ using the trace-fixedness of $\{1-z_1,1-z_2\}$.
	Viewing $\Phi$ as a compactly supported cocycle that is the value of a connecting 
	map on $e$ yields its (pullback) Eisenstein property.\footnote{Xu proves this only after pullback by a different method.}
	
	The map $\Phi$ restricts to a prime-to-level Eisenstein modular symbol from the subgroup
	$\mr{St}_N$ of $\mr{St}$ supported away from $\tilde{\Gamma}_0(N) \cdot \infty$ to a quotient of $\ms{K}_{2,N}$.
	This map $\Phi|_{\mr{St}_N} \otimes \zp$ can be pulled back by $(1,\zeta_N)$ to obtain a 
	modular symbol that factors through 
	$$
		((\St_N)_{\Gamma_1(N)} \otimes_{\Z} \zp)_+ \cong \mc{S}^{\circ},
	$$ 
	using $\tau_N$ in the isomorphism.
	There, the ambiguity from $W$ disappears, and the map induces $\Pi^{\circ}$.
	The map $\Pi^{\circ}$ is Eisenstein for $U_{\ell}$-operators by a computation using the remaining equivariance
	of the modular symbol. 
\end{proof}

A straightforward extension of \cite[Theorem 5.4.5]{fk-sh} to include $p = 3$ (cf.~\cite{xu}) tells us that $\Pi^{\circ}(\mc{S}) \subset Y$. 
Set $P = \mc{S}/I\mc{S}$. Let $\varpi \colon P \to Y$ denote the map induced by $\Pi^{\circ}|_{\mc{S}}$.

\begin{conjecture}[S.] \label{varpi_isom}
	The map $\varpi$ is an isomorphism.
\end{conjecture}

The conjectural surjectivity of $\varpi$ for $N = p$ dates back to work with McCallum in \cite{mcs}.
In writing this article, we used Sage via Codex to check, for $N \le 250$ and all odd $p$, equality of the orders of the maximal quotients on which diamond operators and inverse Galois elements act via a given character of prime-to-$p$ order on $P$ and $Y$ respectively \cite{sh-comp}. This built on Magma routines written to test the conjecture over a decade ago.

The earlier approach of Fukaya and Kato to the Eisenstein property is important for its connection with the geometry of modular curves. The key tool, for $p \mid N$ and $p \nmid \varphi(M)$ with $p \ge 5$, is a zeta map
$$
	z^{\circ} \colon H_1(X_1(N),C'_0,\zp) \to H^2_{\et}(Y_1(N),\qp(2))^{\ord}, \quad z^{\circ}([u:v]) = e^{\mr{ord}}(g_{0,u/N} \cup 
	g_{0,v/N}),
$$
where $g_{0,\frac{u}{N}}$ for $N \nmid u$ is a Siegel unit on $Y_1(N)$ and ``ord'' (resp., $e^{\mr{ord}}$) 
denotes the $U_p$-ordinary part (resp., ordinary idempotent).

In \cite{sv}, we constructed the restriction of $z^{\circ}$ to $H_1(X_1(N),\zp)$ and proved its prime-to-level Hecke-equivariance for $p \ge 7$. There, its integral $c$-stabilization for $c \ge 2$ is obtained from a cocycle ${}_c \Theta \colon \Gamma \to K_2(\Q(E^2)) \otimes_{\Z} \zp$   attached to a $\Gamma$-fixed formal sum of $c$-torsion subschemes, where $E$ is the universal elliptic curve over $Y_1(N)$. This mirrors Kato's construction \cite{kato} of theta-functions ${}_c \theta$ on $E$ as trace-fixed elements with $c$-torsion divisors.

Using an integral stabilization of $z^{\circ}$ and a specialization map $\infty^*$ to $Y^{\circ}$,
Fukaya and Kato recover $\Pi^{\circ}$. Specialization is Hecke-equivariant and Eisenstein away from the level, while
the $U_{\ell}$-Eisenstein property is seen on the cup products of Siegel units. Setting $R = \Z[\mu_N,\frac{1}{N}]$, 
this connects to the method of \cite{sv} via the diagram
\begin{equation*}
	\begin{tikzcd}
		\mb{G}_{m/R}^2 \arrow[r] & \tilde{E}_{/R}^2 \\
		\Spec R \arrow{r}{\infty} \arrow{u}{(1,\zeta_N)} & X_1(N)_{/R} \arrow{u}{(0,\iota_N)},
	\end{tikzcd}
\end{equation*}
where $\iota_N$ is the canonical $N$-torsion section on the universal generalized elliptic curve $\tilde{E}$. 
Here, we view $\gm$ as the smooth locus of the nodal cubic fiber of $\tilde{E}$ over $\infty$.
That is, we can specialize our modular symbol for $E^2$ to that of $\mb{G}_m^2$ by pulling back over a formal neighborhood.  

\section*{From arithmetic to geometry}

\subsection*{The Iwasawa main conjecture}

For an even character $\theta$ of $(\Z/Mp\Z)^{\times}$ and a $\Lambda$-module $A$, viewed as a $\zp[(\Z/Mp\Z)^{\times}]$-module
via the canonical homomorphism $(\Z/Mp\Z)^{\times} \to (\Z/N\Z)^{\times}$, we let 
$$
	A_{\theta} = A \otimes_{\zp[(\Z/Mp\Z)^{\times}]} \mc{O}_{\theta},
$$ 
where $\mc{O}_{\theta}$ is the $\zp$-algebra of $\theta$-values, for the map $\Z[(\Z/Mp\Z)^{\times}] \to \mc{O}_{\theta}$ 
induced by $\theta$. We denote by $\cond(\theta)$ the conductor of $\theta$.

For an abelian group $B$, let $B_{\Q} = B \otimes_{\Z} \Q$.
For any prime $\ell \mid N$ and $q = \ell^{v_{\ell}(N)}$, we define $[\ell]$ as the product of the
idempotent in $\La_{\Q}$ attached to the trivial character of the image of $(\Z/q\Z)^{\times}$ and the group element of $\ell$ in the image of $(\Z/\frac{N}{q}\Z)^{\times}$. We then define $[d]$ for any $d \in \Z - \{0\}$ multiplicatively.
For $B_2(x) = x^2-x+\frac{1}{6}$, set
$$
	\xi = -\frac{N}{2} \sum_{d=1}^N B_2(\tfrac{d}{N})[d] \in \La_{\Q}.
$$
Setting $\xi^{(p)} = (1-p[p])\xi$, the $\xi^{(p)}_r \in (\Lambda_r)_{\Q}$ for $r \ge 1$ are compatible under the canonical quotient maps of group rings and yield an element $\xi^{(p)}_{\infty}$ in the total quotient ring $\mc{Q}$ of $\Lambda_{\infty}$.

Let $X_{\infty}$ denote the Galois group of the maximal abelian unramified pro-$p$ extension of $F_{\infty} = \Q(\mu_{Mp^{\infty}})$.  The Iwasawa main conjecture concerns the $\zp\ps{\Gal(F_{\infty}/\Q)}$-module structure of $X_{\infty}^-$ and so the $\La_{\infty}$-module structure of $X_{\infty}(1)^+$. Let $f_{\theta}$ be a numerator in lowest terms of the projection of $\xi^{(p)}_{\infty}$ to the quotient field $\mc{Q}_{\theta}$ of $\Lambda_{\infty,\theta}$. Mazur and Wiles \cite{mw} proved:

\begin{theorem}[Iwasawa main conjecture] \label{mainconj} 
	Suppose that $\theta$ is an even character of $(\Z/Mp\Z)^{\times}$ of conductor $\cond(\theta)$ divisible by $M$. Then 
	the $\La_{\infty,\theta}$-characteristic ideal of $X_{\infty}(1)_{\theta}$ equals $(f_{\theta})$.
\end{theorem}

\begin{remark}
The inverse limit $Y_{\infty}$ of the groups $Y$ over increasing $r$ is related to $X_{\infty}(1)^+$ via a canonical map
$X_{\infty}(1)^+ \to Y_{\infty}$. In fact, for $\theta$ with $M \mid \cond(\theta)$, the $\Lambda_{\infty,\theta}$-modules 
$Y_{\infty,\theta}$ and $X_{\infty}(1)_{\theta}$ have equal characteristic ideals.
\end{remark}

The work of Mazur and Wiles \cite{mw} also implies that 
\begin{equation*} \label{Hecke_mod_Eis_eig}
	(\mf{h}/I)_{\infty,\theta} \cong \Lambda_{\infty,\theta}/(f_{\theta})
\end{equation*}
under its hypothesis on $\theta$.
This is not accidental, for $\xi^{(p)}_{\theta}$ is closely related to the constant term of the projection of the equivariant Eisenstein series
$$
	\mc{E} = \frac{1}{2}\xi^{\sharp} + \sum_{n=1}^{\infty} \sum_{\substack{d \mid n\\(d,N)=1}} d[d] q^n \in \Lambda_{\Q}\ps{q},
$$
where
$$
	\xi^{\sharp} = -\frac{N}{2} \sum_{\substack{d=1\\(d,N)=1}}^N B_2(\tfrac{d}{N})[d] = \prod_{\ell \mid N} (1-\ell[\ell]) \xi.
$$
In the limit, the $\xi^{\sharp}_r$ give an equivariant $p$-adic $L$-function at $s = -1$ with its Euler factors at $\ell \mid M$ removed.
For $p \ge 5$, Lafferty \cite{lafferty} computed the analogous characteristic ideal for arbitrary even characters, avoiding the main
conjecture.

For an $\mf{H}$-module $A$, let $A_{\mf{m}}$ denote its localization at the complement of the union of the maximal ideals containing
$I$. We let Hecke algebras act on \'etale cohomology groups of modular curves via the contravariant action of Hecke operators.
The semilocalization $\mc{T}_{\mf{m}}$ of
$$
	\mc{T} = H^1_{\et}(X_1(N)_{/\qbar},\zp(1))
$$ 
is an $\mf{h}_{\mf{m}}[G_{\Q}]$-module of generic rank two over $\mf{h}_{\mf{m}}$. 
The geometric method of Ribet-Mazur-Wiles exploits the reducibility of the Eisenstein quotient $\mc{T}/I\mc{T}$. We outline the proof in the form given in Ohta's work \cite{ohta-ord2} for $p \ge 5$.

\begin{proof}[Rough sketch of proof of Theorem \ref{mainconj}]
	Since $\mc{T}_{\infty,\mf{m}}$ is ordinary at $p$,
	there is an $\mf{h}_{\infty,\mf{m}} \otimes_{\La_{\infty}} \mc{Q}_{\theta}$-basis 
	of $\mc{T}_{\infty,\mf{m}} \otimes_{\Lambda_{\infty}} \mc{Q}_{\theta}$
	with respect to which the action of inertia at $p$ is lower-triangular. Let $\rho = \smatrix{a&b\\c&d}$ be the representation
	with respect to this basis. Let $B$ denote the $\mf{h}_{\infty}$-span of the image of $b$. The Eisenstein reduction 
	$\bar{b} \colon G_{\Q} \to B/IB$ of $bd^{-1}$ is a cocycle that restricts to an unramified
	and surjective homomorphism on $G_{F_{\infty}}$.
	The intertwining of Galois and Hecke actions on $B$ shows that $\bar{b}$ provides a surjective map 
	$X_{\infty}(1)_{\theta} \to B/IB$ of 
	$\Lambda_{\infty,\theta}$-modules. Since $B$ is faithful as a $\mf{h}_{\infty,\mf{m},\theta}$-module, 
	the $\La_{\infty,\theta}$-Fitting ideal of $B/IB$
	is divisible by $f_{\theta}$. This and the known vanishing of the $\mu$-invariant ensure divisibility of the characteristic 
	ideal of $X_{\infty}(1)_{\theta}$ by $f_{\theta}$.\footnote{The argument must be modified if residual characters of
	the representation agree on $G_{\qp}$.} As we may replace $M$ by the prime-to-$p$ part of $\cond(\theta)$, this gives one 
	divisibility in the main conjecture for all even characters. 
	A standard argument using the analytic class number formula then forces the opposite divisibility.
\end{proof}

\subsection*{The map $\Upsilon$}

Motivated by the geometric proof of the main conjecture, we construct a map $\Upsilon \colon Y \to P$ in the opposite direction to $\varpi$ out of the Eisenstein quotient $T = \mc{T}/I\mc{T}$.

If $p \mid N$, then we have the canonical composition $\omega \colon (\Z/N\Z)^{\times} \twoheadrightarrow (\Z/p\Z)^{\times} \hookrightarrow \zp^{\times}$
of the quotient map and the lifting map. The original construction from \cite{sh-Lfn} (cf. \cite{fk-sh}) was of a map $\Upsilon_{\theta}$ in a $\theta$-eigenspace for an even character $\theta$ of $(\Z/Mp\Z)^{\times}$ satisfying the following hypotheses:  
\begin{equation} \label{hypotheses}
	p \ge 5, \quad p \mid N, \quad p \nmid \varphi(M), \quad M \mid \cond(\theta), \quad \theta \neq 1, \omega^{-2}, \quad \text{and} \quad \theta\omega(p) \neq 1.
\end{equation}
The latter condition on the associated primitive character of $\theta\omega$, which can be interpreted Galois-theoretically as nontriviality of $\theta\omega$ on $G_{\qp}$, is also the condition of residual distinguishability mentioned in the previous footnote.

In this section, we explain the construction of $\Upsilon$ equivariantly for any odd prime $p$ and level $N \ge 5$. The full details will be provided in a paper in preparation with Gyujin Oh  \cite{os}. The construction relies upon two dualities and an isomorphism. 

\begin{itemize}
	\item \textbf{Hida duality}\ \ 
	Let $\mf{M}^{\circ}$ denote the $\mf{H}$-submodule of $\mf{M}_{\Q}$ consisting of the modular forms $f$ with integral 
	Fourier coefficients $a_n(f)$ for $n \ge 1$. Hida \cite[Theorem 7.3.5]{hida} defined a perfect $\La$-balanced pairing
	$$
		\mf{H} \times \mf{M}^{\circ} \to \zp, \quad (T,f) \mapsto a_1(Tf),
	$$
	which induces a perfect pairing $\mf{h} \times \mf{S} \to \zp$.  Associating to $f \in \mf{M}^{\circ}$ the expression 
	$\sum_{d \in D} (\langle d \rangle^{-1} f) [d] \in \La_{\Q}\ps{q}$ provides an $\mf{H}$-module of such $\La$-adic
	forms that we identify with $\mf{M}^{\circ}$. The pairings 
	extend to perfect $\La$-module pairings $\mf{H} \times \mf{M}^{\circ} \to \La$ and $\mf{h} \times \mf{S} \to \La$.

	\item \textbf{Twisted Poincar\'e duality}\ \ 
	We have a modified form of Poincar\'e duality, related to Ohta's construction in \cite{ohta-es}. 
	Let $\La^{\iota}$ be the $\La[G]$-module on which $\sigma_a$ acts as $[a]$.
	The cup product of Poincar\'e duality on $\mc{T}$ yields a perfect, $\La$-bilinear pairing 
	$$
		(\ , \ ) \colon \mc{T}_{\mf{m}} \times \mc{T}_{\mf{m}} \to \La^{\iota}(1), \quad (x,y) 
		= \sum_{d \in (\Z/N\Z)^{\times}} (x \cup \langle d \rangle \tau_N y)[d]
	$$
	with the properties $(Tx,y) = (x,Ty)$ for $T \in \mf{h}$, and $(\sigma x, \sigma y) 
	= \sigma(x,y)$ for $\sigma \in G_{\Q}$.

	\item \textbf{Ordinary Eichler-Shimura}\ \ 
	Ohta's $p$-ordinary Eichler-Shimura theorem \cite{ohta-es,ohta-ord2} provides a 
	canonical exact sequence
	\begin{equation} \label{ord_ex_seq}
		0 \to \mc{T}_{\sub} \to \mc{T}_{\mf{m}} \to \mc{T}_{\quo} \to 0
	\end{equation}
	of $\mf{h}_{\mf{m}}[G_{\qp}]$-modules and isomorphisms $\mc{T}_{\quo} \cong \mf{S}_{\mf{m}}$ and 
	$\mc{T}_{\sub} \cong \mf{h}_{\mf{m}}$ of $\mf{h}_{\mf{m}}$-modules, the latter isomorphism 
	following from the above two dualities. 
	The $G_{\qp}$-action on $\mc{T}_{\quo}$ is unramified, with $U_p$ acting as the arithmetic Frobenius $\Fr_p$ if $p \mid N$,
	and the above isomorphisms become canonical upon Eisenstein reduction.
	Ohta supposes that $p \mid N$ and $p \ge 5$, but Vu \cite{vu} outlined an argument for
	$p = 3$. The $p \nmid N$ case follows via ordinary $p$-stabilization, identifying 
	$\mc{T}_{0,\mf{m}}$ with the ordinary $p$-old part of $\mc{T}_{1,\mf{m}}$.
\end{itemize}

First, we wish to employ Hida duality to construct an isomorphism 
\begin{equation} \label{Hecke_mod_Eis}
	\mf{h}/I \xrightarrow{\sim} \La/\mc{L}
\end{equation} 
for a certain nonzero ideal $\mc{L}$ of $\La$.

\begin{proof}[Detail of \eqref{Hecke_mod_Eis}]
	Let $\mc{C} = H_0(C_1(N),\zp)$, where $C_1(N)$ denotes the cusps of $X_1(N)$. 
	We have an $\mf{H}$-equivariant residue map $\Res \colon \mf{M}_{\Q} \to \mc{C} \otimes_{\zp} L$ 
	for some finite $L/\qp$, given by
	$$
		\Res(f) = \sum_{\alpha \in C_1(N)} w_{\alpha} a_{0,\alpha}(f) \tau_N(\alpha)
	$$ 
	for $w_{\alpha}$ the width of $\alpha$ and $a_{0,\alpha}(f)$ the constant coefficient at $\alpha$. 
	Write $\mc{E} = \sum_{d \in D} E_d [d]$ with $E_d \in \mf{M}^{\circ}$, and set $\kappa = \Res(E_1)$.
	The $I$-torsion submodule $\mf{M}^{\circ}[I]$ is the free $\La$-module on $\mc{E}$, and one can show that 
	$\mc{C}[I]_{\Q} = \La_{\Q} \kappa$.

	Let $\mf{M}' = \{ f \in \mf{M} \mid \Res(f) \in \mc{C}[I]\}$, and set $\mf{M}^{\sharp} = \mf{M}' + \La \mc{E}$.
	Let $\mc{C}^{\sharp} = \Res(\mf{M}^{\sharp}) = \mc{C}[I] + \La \kappa$. 
 	We then have an exact sequence of $\mf{H}$-modules
	\begin{equation} \label{sharp_mod_seq}
		0 \to \mf{S} \to \mf{M}^{\sharp} \to \mc{C}^{\sharp} \to 0.
	\end{equation}
	A result of Katz \cite[Key Lemma 6.1.9]{katz} ensures that $\Res(\mf{M}') = \mc{C}[I]$ and the perfectness of the
	Hida pairing between $\mf{M}^{\sharp}$ and the image 
     	of $\mf{H}$ in $\End_{\La}(\mf{M}^{\sharp})$.
        Applying this pairing  to $\mc{E}$ gives an isomorphism 
        $\mf{H}/I \xrightarrow{\sim} \La$. Dualizing \eqref{sharp_mod_seq} now yields \eqref{Hecke_mod_Eis} for some $\mc{L}$.
        
        We can describe $\mc{L}$ more precisely. Consider the map $\pi_0 \colon \mc{C}_{\Q} \to \La_{\Q}$ taking a sum of cusps 
	to its subsum of zero cusps and identifying the cusp $\frac{1}{d}$ with $[d]$ for each $d \in D$. 
	It is an isomorphism on $\mc{C}[I]_{\Q}$ and satisfies $\pi_0(\kappa) = a_{0,\infty}(\mc{E}) = \frac{1}{2}\xi^{\sharp}$.
	By Hida duality, $\mathcal L$ is the $\La$-ideal of $F(\kappa)$ with $F \in \Hom_{\Lambda}(\mathcal{C}^{\sharp},\La)$.
	Under the identification provided by $\pi_0$, any such $F$ becomes multiplication by some $u\in\Lambda_{\mathbb Q}$ 
	such that $u\pi_0(\mathcal C^\sharp)\subseteq\Lambda$, and $F(\kappa) = \frac12\xi^\sharp u$. Thus, we obtain
	$$
		\mc{L} = \tfrac{1}{2}\xi^{\sharp}\pi_0(\mc{C}^{\sharp})^{-1} = 
		\tfrac{1}{2}\xi^{\sharp} \pi_0(\mc{C}[I])^{-1} \cap \La,
	$$ 
	where the inverse means the inverse fractional ideal of $\La$. This yields Lafferty's computation \cite{lafferty} 
	of the characteristic ideal of $(\mf{h}/I)_{\infty,\theta}$ for each $\theta$.
 \end{proof}
        
Now let us begin the construction of $\Upsilon$. Let $Q$ denote the $\mf{h}[G_{\Q}]$-module $(\mf{h}/I)^{\iota}(1)$, where $(\mf{h}/I)^{\iota}$ is the $\mf{h}$-module $\mf{h}/I$ with $G_{\Q}$-action factoring through $G$ such that $\sigma_a$ acts as $\langle a \rangle$.

\begin{step} \label{canon_surj}
	Define a canonical surjection $\phi \colon T \to Q$ of $\mf{h}[G_{\Q}]$-modules.
\end{step}

\begin{proof}[Construction of $\phi$]
	We first modify the \'etale cohomology group
	$$
		\tilde{\mc{T}} = H^1_{\et}(Y_1(N)_{/\qbar},\zp(1)) \cong H_1(X_1(N),C_1(N),\zp)
	$$
	analogously to how $\mf{M}^{\sharp}$ is defined from $\mf{M}$. The $I$-torsion in the cokernel of $\mc{T} \to \tilde{\mc{T}}$ 
	may be identified with $\mc{C}[I]$. We take $\tilde{\mc{T}}^{\sharp}$ to be the sum of $\Res^{-1}(\mc{C}[I])$ and the $\La$-span 
	of $t(\kappa)$ for the Manin-Drinfeld splitting $t \colon \mc{C}[I]_{\Q} \to \tilde{\mc{T}}_{\Q}$.
	Since zero cusps are $G_{\Q}$-fixed, for each $\sigma \in G_{\Q}$, the difference $(\sigma-1)t(\kappa)$ lies in $\mc{T}_{\Q}[I] = 0$.
	That is, $t(\kappa)$ is $G_{\Q}$-invariant.
	In fact, it is the Kummer class of $g^{-1}$ for $g = g_{0,\frac{1}{N}}$.  Since $\lambda g \in \tilde{\mc{T}}$ for 
	$\lambda \in \La$ if and only if $\lambda \mc{E} \in \mf{M}$, the sequence
	\begin{equation} \label{cohom_sharp_seq}
		0 \to \mc{T} \to \tilde{\mc{T}}^{\sharp} \to \mc{C}^{\sharp} \to 0
	\end{equation}
	of $\mf{H}[G_{\Q}]$-modules is exact, where the $G_{\Q}$-action on $\mc{C}^{\sharp}$ is trivial.

	Let $\mc{B}$ denote the reduced cohomology group $\bar{H}^0_{\et}(C_1(N)_{\qbar},\zp(1))$ so that
	$$	
		\mc{T}_c = H^1_{c,\et}(Y_1(N)_{/\qbar},\zp(1)) \cong H_1(Y_1(N),\zp)
	$$ 
	has $\mc{T}$ as a quotient with kernel $\mc{B}$.
	Twisted Poincar\'e duality extends to a pairing $\tilde{\mc{T}} \times \mc{T}_c \to \La^{\iota}(1)$.
	Setting $\mc{B}^{\sharp} = \Hom_{\La}(\mc{C}^{\sharp},\La^{\iota}(1))$ and 
	$\mc{T}_c^{\sharp} = \Hom_{\La}(\tilde{\mc{T}}^{\sharp},\La^{\iota}(1))$, 
	we have an exact sequence 
	$$
		0 \to \mc{B}^{\sharp} \to \mc{T}_c^{\sharp} \to \mc{T} \to 0,
	$$
	where here $\mc{T} \cong \Hom_{\La}(\mc{T},\La^{\iota}(1))$
	via left twisted pairing with elements of $\mc{T}$.
	The surjection $\Phi_c \colon \mc{T}_c^{\sharp} \to \La^{\iota}(1)$ of
	$\mf{H}[G_{\Q}]$-modules given by evaluation on $t(\kappa)$
	has image $\mc{L}$ on $\mc{B}^{\sharp}$,
	so it induces a map $\mc{T} \to (\La/\mc{L})^{\iota}(1)$ which by \eqref{Hecke_mod_Eis} is exactly the desired surjection
	$\phi \colon T \to Q$ of $\mf{h}[G_{\Q}]$-modules.
\end{proof}

\begin{remark} \label{integral_Siegel}
	In Step \ref{canon_surj}, we used the known property of $g = g_{0,\frac{1}{N}} \in \tilde{\mc{T}}_{\Q}$ 
	that $\lambda g \in \tilde{\mc{T}}$ for 
	$\lambda \in \La$ if and only if $\lambda \xi^{\sharp} \in \La$. Here is one way to see the less straightforward ``if'' implication 
	using Betti-\'etale comparison. The ideal of $\Lambda$ making $\xi^{\sharp}$ integral is 
	$$
		J = (1-c^2[c] \mid (c,6Mp) = 1).
	$$ 
	Let $\mb{T} = \R^2/\Z^2$, with $\Gamma_1(N)$ acting on the left. 
	As in \cite[\S 3.2]{bcg}, there is a unique trace-fixed away from $c$, equivariant class
	$m_c \in H^1_{\Gamma_1(N)}(\mb{T}-\mb{T}[c],\zp)$ with residue $\mb{T}[c]-c^2(0)$ for $c$ prime to $Mp$. 
	Set $v = (\frac{1}{N},0) \in \mb{T}[N]$. Pulling back $m_c$ by 
	$c^{-1}v$ yields an element $\upsilon_c \in H^1(\Gamma_1(N),\zp) = H^1(Y_1(N)(\C),\zp)$.
	At the same time, \cite[\S 2]{xu-eis} (cf.~\cite[Theorem 21]{bcg}) identifies a class $\vartheta \in 
	H^1(\Gamma_1(N),\Q)$ that is Eisenstein away from the level, has 
	boundary in $H_0(C_1(N),\qp)$ equal to $\kappa$, and satisfies $\upsilon_c = (1-c^2\langle c \rangle)\vartheta$.
	The boundary map is injective on the away-from-$N$ Eisenstein part, so 
	$\vartheta$ corresponds to $-g$ under integral Betti-\'etale comparison, 
	and therefore $\upsilon_c$ corresponds to $(c^2[c]-1)g \in \tilde{\mc{T}}$.
\end{remark}

The exact sequence of ordinary $\mf{H}[G_{\qp}]$-quotients of \eqref{cohom_sharp_seq} is isomorphic to 
the ordinary part of
$$
	0 \to \mf{S} \to \mf{M}^{\sharp} \to \mc{C}^{\sharp} \to 0.
$$
Eichler-Shimura on $I$-torsion gives a canonical isomorphism $\tilde{\mc{T}}[I]_{\Q} \xrightarrow{\sim} \mf{M}[I]_{\Q}$ that sends $t(\kappa)$ to $\mc{E}$. Thus, we can think of $\phi$ as being induced by pairing with $\mc{E}$.
One can also construct $\phi$ as an $\mf{h}$-module map by integrating against $\mc{E}$.

\begin{proof}[The map $\phi$ as integration]
For $\mc{N} = H_1(Y_1(N),\zp)$, we have an exact sequence
$$
	0 \to \mc{B} \to \mc{N} \to H_1(X_1(N),\zp) \to 0.
$$
Betti-\'etale comparison identifies $\mc{N}$ with $\mc{T}_c$, and the analogous construction yields $\mc{N}^{\sharp} \cong \mc{T}_c^{\sharp}$.
The map
$$
	\Phi_{\mr{dR}} \colon \mc{N}^{\sharp} \to (\C \otimes_{\Z} \zp)[D], \quad
	\Phi_{\mr{dR}}(\gamma) = -2\int_{\gamma} \mc{E}(z) dz
$$
is valued in the image of $\La_{\Q}$ since
\begin{equation} \label{Eis_Siegel}
	2\pi i E_1(z)dz = -\dlog(g(\tau_N z)).
\end{equation}
The maps $\Phi_{\mathrm{dR}}$ and $\Phi_c$ agree on $\mathcal B^\sharp$, as the integrals of $E_d(z)dz$ along simple positively oriented loops about cusps are residues of $E_d$, with the factor of $-2$ arising from the ordering of and summation indices in the twisted Poincar\'e pairing. The difference $\Phi_{\mathrm{dR}} - \Phi_c$ on $\mc{N}^{\sharp}$ then factors through an Eisenstein map 
$H_1(X_1(N),\zp) \to \La_{\Q}$, which must be zero.
\end{proof}

For the next step of the construction, we let $K = \ker \phi$ so that we have an exact sequence
of $(\mf{h}/I)[G_{\Q}]$-modules 
$$
	0 \to K \to T \xrightarrow{\phi} Q \to 0
$$
with $K^+ \cong P$.  We write $T_{\sub} = \mc{T}_{\sub}/I\mc{T}_{\sub}$ and $T_{\quo} = \mc{T}_{\quo}/I\mc{T}_{\quo}$. Ordinary Eichler-Shimura and Hida duality tell us that $T_{\sub} \to Q$ is a surjective map of free $\mf{h}/I$-modules of rank one, which must then be an isomorphism. Consequently, $\phi$ is split as a map of $\mf{h}[G_{\qp}]$-modules by the inverse $s$ of $\phi|_{T_{\sub}}$, and $K \to T_{\quo}$ is an isomorphism as well. 

In any $\theta$-quotient of $\Lambda$ for a character $\theta$ of $(\Z/Mp\Z)^{\times}$ such that $\theta\omega(p) \neq 1$ in the residue field of $\mc{O}_{\theta}$, i.e.,  for which the two residual characters of $\mc{T}_{\mf{m},\theta}$ are distinct on $G_{\qp}$,
the $\theta$-quotient of the ordinary exact sequence \eqref{ord_ex_seq} is $\mf{h}_{\theta}$-split, and in fact, $(\mc{T}_{\mf{m},\theta})^+ \to \mc{T}_{\quo,\theta}$ is an isomorphism so that $K_{\theta} = P_{\theta}$. However, the unreduced ordinary sequence does not split in general. In fact, a computation using Codex \cite{os} showed that for $p = 5$, $N = 2995$, and $\theta\omega$ the quadratic character attached to $\Q(\sqrt{-599})$, the $\theta$-eigenspace of \eqref{ord_ex_seq} is not $\mf{h}_{\theta}$-split, and the minus part $K^-_{\theta}$ is nonzero and cyclic.

Though the splitting of \eqref{ord_ex_seq} is heavily used in earlier work, we can actually avoid it.

\begin{step} 
	Show that there is an exact sequence of $(\mf{h}/I)[G_{\Q}]$-modules 
	\begin{equation} \label{keyseq}
		0 \to P \to T' \xrightarrow{\phi'} Q \to 0,
	\end{equation}
	for a subquotient $T'$ of $T$, where $\phi'$ is induced by $\phi$. 
\end{step}

\begin{proof}[Sketch of proof.]
	Let $\rho \colon G_{\Q} \to \Aut_{\La}(T)$ be given by the Galois action. Write $\rho$ as a generalized matrix representation
	$\smatrix{a&b\\c&d}$ with respect to the decomposition $T = P \oplus T^-$, using the identification $T^+ \cong P$. 
	Let $\delta \colon G_{\Q} \to (\mf{h}/I)^{\times}$
	denote the Galois character of $Q$.
	From the relation 
	$$
		\rho(\Fr_{\ell})^2 - T_{\ell} \rho(\Fr_{\ell}) + \ell \langle \ell \rangle = 0
	$$ 
	on a Frobenius $\Fr_{\ell}$ at $\ell \nmid Mp$ 
	and the identity $\rho(\varsigma) = \smatrix{1&0\\0&-1}$, one may check that $a = 1$, $d = \delta$, and
	$b(\sigma)c(\tau) = c(\tau)b(\sigma) = 0$ for all $\sigma, \tau \in G_{\Q}$.
	
	Let 
	$$
		C = \sum_{\sigma \in G_{\Q}} c(\sigma)(P) \subseteq K^-.
	$$ 
	The injective image of $P$ in $T/C$ is a trivial $G_{\Q}$-submodule. Since $d = \delta$, Galois acts on $T/(P + C)$ via
	$\delta$. Since $\phi$ is trivial on the $\zp[G_{\Q}]$-span $P + C$ of $P$, the map $\phi$ factors through $T/(P + C)$.
	Since $G_{\Q}$ acts on $Q$ and $T/(P+C)$ through $\delta$, the induced splitting $Q \to T/(P + C)$ 
	is $G_{\Q}$-equivariant. Thus $T' = (P + C + s(Q))/C$ is a $G_{\Q}$-subquotient of $T$ with the desired properties.
\end{proof}

The $\La$-module $Y$ is canonically isomorphic to $H^2_{\et}(\Z[\mu_N,\tfrac{1}{p}],\zp(2))_+$ by a result of Tate \cite{tate}. 
Let $\Gamma = \Gal(F_{\infty}/\Q(\mu_N)^+)$.
The canonical map $X_{\infty}(1)_{\Gamma} \to Y$ is an isomorphism in any $\theta$-quotient with $\theta\omega(p)$
not residually equal to $1$.
The exact sequence \eqref{keyseq} provides the class of a cocycle
$$
	b' \colon G_{\Q} \to \Hom_{\mf{h}}(Q,P), \quad b'(\sigma)(q) = \sigma s'(\sigma^{-1}q) - s'(q),
$$
where $s' \colon Q \to T'$ is induced by $s$. 
This cocycle is unramified outside $Mp\infty$ and restricts to an unramified homomorphism on $G_{F_{\infty}}$, inducing a $\La$-module homomorphism $X_{\infty}(1)_{\Gamma} \to P$ after evaluation at $1 \in Q$. This provides the desired map $\Upsilon_{\theta}$ on eigenspaces satisfying \eqref{hypotheses}, as in \cite{sh-Lfn, fk-sh}. However, under \eqref{hypotheses} except that $\theta\omega(p) =1$, Shih and Wang \cite{shihwang} proved that $X_{\infty}(1)_{\theta} \to P_{\infty,\theta}$ is zero after localizing at the prime of $\La_{\infty,\theta}$ giving the trivial zero.

\begin{step}
	Construct a $\La$-module homomorphism $\Upsilon \colon Y \to P$ as a connecting map in an exact sequence derived
	from \eqref{keyseq}.
\end{step}

\begin{proof}[Outline of construction.]
	Pull back \eqref{keyseq} along the surjection $\La^{\iota}(1) \to Q$ to obtain an exact sequence
	\begin{equation} \label{pullkey}
		0 \to P \to \tilde{T} \xrightarrow{\tilde{\phi}} \La^{\iota}(1) \to 0.
	\end{equation}
	For a prime $\ell \mid M$ and $\sigma, \tau \in G_{\Q_{\ell}}$ with $\tau$ in the inertia subgroup $I_{\ell}$, 
	one has $(\rho(\sigma)-\delta(\sigma))(\rho(\tau)-1) = 0$. One can apply this with $\sigma = \Fr_{\ell}$ and $\tau$ a lift of a 
	topological generator of the pro-$p$ tame quotient of $I_{\ell}$ to show that $\tilde{\phi}$ remains surjective 
	on $I_{\ell}$-invariants for $\ell \mid M$, and in fact modulo any power of $p$.
	Therefore, the pushforward of \eqref{pullkey} by $j \colon \Spec \Z[\frac{1}{Mp}] \to \Spec \Z[\frac{1}{p}]$ is exact.
	The map $\tilde{\phi}$ has a $\La[G_{\qp}]$-module splitting $\tilde{s}$ induced by $s'$.
	Let $\tilde{t} = \id - \tilde{s} \circ \tilde{\phi} \colon
	\tilde{T} \to P$ be the $\La[G_{\qp}]$-module retraction induced by $\tilde{s}$.
	We can define a Selmer complex $R\Gamma_{f,\et}(\Z[\frac{1}{p}],j_*\tilde{T}(1))$ as the shift by $-1$ of the cone of the composition
	$$
		R\Gamma_{\et}(\Z[\tfrac{1}{p}],j_*\tilde{T}(1)) \xrightarrow{ \tilde{t}\, \circ\, \res_p } R\Gamma_{\et}(\qp,P(1)),
	$$ 
	where $\res_p$ denotes the localization map at $p$. It fits in a distinguished triangle
	$$
		R\Gamma_{c,\et}(\Z[\tfrac{1}{p}],P(1)) \to 
		R\Gamma_{f,\et}(\Z[\tfrac{1}{p}],j_*\tilde{T}(1)) \to R\Gamma_{\et}(\Z[\tfrac{1}{p}],j_*\La^{\iota}(2))
		\to R\Gamma_{c,\et}(\Z[\tfrac{1}{p}],P(1))[1].
	$$
	The second connecting map in the corresponding long exact sequence gives
	$$
		\Upsilon \colon Y \xrightarrow{\sim} H^2_{\et}(\Z[\tfrac{1}{p}],j_*\La^{\iota}(2)) \to H^3_{c,\et}(\Z[\tfrac{1}{p}],P(1)) 
		\xrightarrow{\sim} P,
	$$
	where the first isomorphism uses Shapiro's
	lemma and corestriction, and the last isomorphism is the invariant map of Poitou-Tate duality.
\end{proof}

We make the following conjecture.

\begin{conjecture} \label{Upsilon_isom}
	The map $\Upsilon \colon Y \to P$ is an isomorphism.
\end{conjecture}

Under \eqref{hypotheses}, the following results are known.
Using the theory of pseudo-representations, Wake and Wang-Erickson \cite{wwe} showed that if $Y_{\infty,\omega^{-2}\theta^{-1}}$ is finite, then $\Upsilon_{\infty,\theta}$ is injective with finite cokernel, and moreover
that it is an isomorphism upon localizing away from the support of $Y_{\infty,\omega^{-2}\theta^{-1}}$ under the involution $[a] \mapsto a^{-2}[a]^{-1}$.
Ohta \cite{ohta-mu, ohta-CM} proved that if the induced character on $(\Z/p\Z)^{\times} \hookrightarrow (\Z/N\Z)^{\times}$ has nontrivial kernel, then the map $\Upsilon_{\theta}$ is an isomorphism; to do so, he used the Zariski density of CM points in products of modular varieties in characteristic $p$ to prove the vanishing of $\mu$-type subgroups of $J_1(Mp)$.

\subsection*{Combining the ingredients}

The following conjecture, which subsumes Conjectures \ref{varpi_isom} and \ref{Upsilon_isom}, is an equivariant extension of the central conjecture in the theory (cf.~\cite[Conjecture 4.12]{sh-Lfn}). In the inverse limit over $r$, it may be viewed as an explicit refinement of the Iwasawa main conjecture.

\begin{conjecture} \label{key_conj}
	The maps $\Upsilon$ and $\varpi$ are inverse to each other.
\end{conjecture}

Let us return to known results and work under \eqref{hypotheses}. Let $\chi \colon \zp^{\times} \to \zp$ denote the character $(1-p^{-1})\log_p$, let $t \in 1+p\zp$ with $\chi(t) = 1$, and set $X = [t]-1$ so that $\La_{\infty,\theta} = \mc{O}_{\theta}\ps{X}$. Let $\xi'_{\theta}$ denote the image in $\La_{\theta}$ of $\xi'_{\infty,\theta} = (1+X) \frac{d}{dX}\xi_{\infty,\theta}^{(p)}$. The following identity of Fukaya and Kato is a slight modification of \cite[Theorem 7.2.3(1)]{fk-sh} found in \cite[Theorem 5.21]{sh-fk}.

\begin{theorem} \label{fkident}
	We have $\xi'_{\theta} \Upsilon_{\theta} \circ \varpi_{\theta} = \xi'_{\theta}$ as 
	endomorphisms of $P_{\theta}$.
\end{theorem}

An equivariant extension of Theorem \ref{fkident} should hold via the Fukaya-Kato method, having the form that for every $g \in \mc{L}_{\infty}$, one has $\frac{dg}{dX} \Upsilon_{\infty} \circ \varpi_{\infty} = \frac{dg}{dX}$. We sketch a few features of the proof of the original result at finite level. See \cite{fks-iwa} for a more thorough exposition.

\begin{proof}[Notes on the proof of Theorem \ref{fkident}.]
	We start with the Eisenstein reduction of Hochschild-Serre applied to the integral restriction \cite[Theorem 5.19]{sh-fk} 
	of the Hecke-equivariant zeta map $z^{\circ}$ to the $\theta$-eigenspace of $H_1(X_1(N),\zp)$, which has the form
	$\bar{z}_{\theta} \colon P_{\theta} \to H^1_{\et}(\Z[\tfrac{1}{Mp}],T_{\theta}(1))$.
	To prove Theorem \ref{fkident}, one compares two compositions of maps that start with $\bar{z}_{\theta}$.
	First, we have $\phi_{\theta} \circ \bar{z}_{\theta} = \varpi_{\theta}$ under an
	identification $H^1_{\et}(\Z[\tfrac{1}{Mp}],Q_{\theta}(1)) \cong Y_{\theta}$ that comes from viewing $Q_{\theta}$ as the
	cokernel of multiplication by $\frac{1}{2}\xi_{\theta}^{(p)}$ on $\La_{\theta}^{\iota}(1)$. The connecting
	map for \eqref{keyseq} gives a map $Y_{\theta} \to H^2_{c,\et}(\Z[\tfrac{1}{Mp}],P_{\theta}(1)) \cong P_{\theta}$
	that agrees with $-\frac{1}{2}\xi'_{\theta}\Upsilon_{\theta}$.
	
	Second, we compose the $G_{\qp}$-restriction of $\bar{z}_{\theta}$ with the map induced by the local splitting
	$T_{\theta} \to P_{\theta}$. We follow that by a finite-level Coleman map $\mr{Col}^{\flat} \colon 
    	H^1_{\et}(\qp,P_{\theta}(1)) \to P_{\theta}$ induced by a modified $p$-adic logarithm on $\zp^{\times}$. The result of a sophisticated
	$p$-adic regulator computation on Beilinson-Kato elements that first takes place up the cyclotomic and level towers before 
	descending tells us that the composition just described is multiplication by $-\frac{1}{2}\xi'_{\theta}$.
	This equals the first composition $-\frac{1}{2}\xi'_{\theta}\Upsilon_{\theta} \circ \varpi_{\theta}$ by a diagram chase.
\end{proof}

\begin{remark}
	In \cite{sh-Lfn}, we conjectured that $\Upsilon_{\theta} \circ \varpi_{\theta}$ is
	multiplication by a canonical unit $u_{\theta} \in (\mf{h}/I)_{\theta}^{\times}$. Uncertain of normalizations, 
	the author removed speculation that $u_{\theta} = 1$ in the final version.
	When Fukaya and Kato proved Theorem \ref{fkident}, we gladly made the stronger conjecture. 
	However, in writing \cite{sh-fk}, we uncovered sign issues in \cite{fk-sh} which meant that $u_{\theta} = -1$. 
	As Kato had pointed out, we did not make clear in \cite{sh-Lfn} on which side one was taking a cup product.
	Thus, we made the opposite choice to \cite{fk-sh}, changing the sign of $\Upsilon_{\theta}$.
	
	We recently realized that the value of the $p$-adic regulator on the right-hand side of the Fukaya-Kato identity 
	in the normalization of \cite{fk-sh} is $\frac{1}{2}\xi'_{\theta}$, instead of $\xi'_{\theta}$. 
	In that $\mc{E}_{\theta}$ has constant coefficient $\frac{1}{2} \xi^{\sharp}_{\theta} = \frac{1}{2}\xi^{(p)}_{\theta}$, 
	rather than $\xi^{\sharp}_{\theta}$, the map $\Upsilon_{\theta}$ of this paper is twice that of \cite{sh-fk}. Thus, a factor 
	of $\frac{1}{2}$ appears on both sides of the corrected identity.
\end{remark}

\section*{Euler systems from Eisenstein series}

\subsection*{Modular constructions of cyclotomic units}

We turn to the geometric construction of the Euler system of cyclotomic units, completing the circle of ideas in our program. There is a second global extension, constructed from a relative \'etale cohomology group of $Y_1(N)$, that we discovered through a shadow of it in \cite[\S 9.6]{fk-sh}. Its class is the real projection of the leading coefficient $1-\zeta_N$  of the Siegel unit 
$$
	g = q^{\frac{1}{12}} (1-\zeta_N) \prod_{n=1}^{\infty} (1-q^n \zeta_N)(1-q^n\zeta_N^{-1}) \in \mc{O}_{Y_1(N)}^{\times} \otimes_{\Z} \Q.
$$
The Euler system relations among these classes as $N$ varies can be read off from the divisors of $g$. This was the subject of unpublished work of Wake and the author \cite{shwa} begun in 2015. 

In 2016, Wake and I discovered that something very close to this construction long predated our work.
In weight two, the class was studied by Anderson (unpublished), Harder \cite{harder}, and Brinkmann \cite{brinkmann}. Nevertheless, our emphasis on the geometric construction of an Euler system appeared new.

The existence of the Siegel unit $g$ on $Y_1(N)$ is a seemingly rare circumstance of a motivic object pulling back to an Euler system, while Eisenstein series are plentiful. In our setting, the Eisenstein series $\mc{E}$ is determined by its residue $\kappa = \Res(E_1)$, which by \eqref{Eis_Siegel} is the negative of the divisor of $g$. With an eye towards generalizations, we thereby came to the question of how to extract the class of $1-\zeta_N$ directly from $\mc{E}$, obscuring the role of $g$.

Wake and I did not answer this question satisfactorily at the time. The author shared ideas from our construction with Skinner, who saw how to use $p$-adic Hodge theory to obtain nontrivial integral classes from Eisenstein series. In \cite{skinner}, he explained how our construction is amenable to this approach. Sangiovanni Vincentelli and Skinner refined this method to construct more general Euler systems \cite{sask}. The full details of one such construction for modular curves may be found in the preprint \cite{sst} of Shang, Skinner, and Trejos Su\'arez.

\subsection*{Our approach}
 
We outline an equivariant construction of $1-\zeta_N$ from $\mc{E}$, providing an updated take on \cite{shwa}.

\setcounter{step}{0}
\renewcommand{\theHstep}{cyclotomic.\arabic{step}}
\begin{step} Construct a partially supported cohomology group. \end{step}

\begin{proof}[Outline of construction.]
Consider the relative, or partially compactly supported, cohomology group 
$$
	\breve{\mc{T}}_c = H^1_{\et}(Y_1(N)_{\infty/\qbar}, C_{\infty/\qbar},\zp(1)) = H^1_{\et,C_{\infty}}(Y_1(N)_{/\qbar},\zp(1))
$$ 
for the closed subscheme $C_{\infty}$ of $X_1(N)$ of cusps over $\infty \in X_0(N)$ and 
the union $Y_1(N)_{\infty} = Y_1(N) \cup C_{\infty}$.
Note that $\bar{H}^0_{\et}(C_{\infty/\qbar},\zp(1))$ 
is isomorphic to $\bar{\La}^{\iota}(1)$, where $\bar{\La}$ denotes the quotient of $\La$ by its norm subgroup.
Setting $\breve{\mc{T}} = H^1_{\et}(Y_1(N)_{\infty/\qbar},\zp(1))$, we have an
exact sequence
\begin{equation} \label{part_cpt_supp_seq}
	0 \to \bar{\La}^{\iota}(1) \to  \breve{\mc{T}}_c \to \breve{\mc{T}} \to 0.
\end{equation}
The group
$\breve{\mc{T}}$ fits in its own short exact sequence, containing $\mc{T}$ with quotient the reduced homology group $\bar{H}_0(C'_{\infty},\zp)$ of the non-infinity cusps $C'_{\infty}$. 
\end{proof}

We might next pull back \eqref{part_cpt_supp_seq} by a map $\Lambda \to \breve{\mc{T}}$ sending $1$ to a $G_{\Q}$-fixed class to obtain a class in $H^1_{\et}(\Z[\frac{1}{Mp}],\bar{\La}^{\iota}(1))$. However, the class $-t(\kappa)$ is $G_{\Q}$-fixed but not even in $\breve{\mc{T}}_{\Q}$, since its divisor has coefficient $\frac{1}{12}$ at every infinity cusp. Moreover, $H^1_{\et}(\Z[\frac{1}{Mp}],\bar{\La}^{\iota}(1))$ is not the $p$-completion of the $Mp$-unit group of $\Q(\mu_N)^+$ but its quotient by $\Z[\frac{1}{Mp}]^{\times} \otimes \zp$. 
The first issue can be dealt with by replacing $Y_1(N)$ with the punctured line bundle $\mc{O}_{Y_1(N)_{\infty}}(C_{\infty})^{\times}$ over $Y_1(N)_{\infty}$. The second issue can be handled by taking cohomology of $Y_1(N) \cup C_0$ relative to $C_0$. However, doing so creates the new issue that the relative cohomology has trivial
residue maps on zero cusps, which we can again remedy by using a product of punctured line bundles for the individual zero cusps.

\begin{step} Modify the extension.\end{step}

\begin{proof}[Outline of construction.]
For a line bundle $L$, let $L^{\times}$ denote the bundle with its zero section removed.
Let $Y' = Y_1(N)_{\infty} \cup C_0$, where $C_0$ denotes the cusps over $0$ on $X_0(N)$, and let 
$$
	\tilde{Y}' = \mc{O}_{Y'}(C_{\infty})^{\times} \times_{Y'} \prod_{x \in C_0} \mc{O}_{Y'}(\{x\})^{\times}.
$$ 
For $x \in C_{\infty} \cup C_0$, let $t_x$ be the formal parameter  of the corresponding Tate curve. 
We lift $x$ to $\tilde{x} \in \tilde{Y}'$ by $t_x^{-1}$ on the factor corresponding to the divisor containing $x$ and the canonical
section on the others.
Let $\tilde{C}_{\infty}$ and $\tilde{C}_0$ be the resulting lifts of $C_{\infty}$ and $C_0$, respectively.
Set 
\begin{eqnarray*}
	\breve{\mc{T}}_{\infty,0} = H^1_{\et}(\tilde{Y}'_{/\qbar} , \tilde{C}_{\infty/\qbar} \cup \tilde{C}_{0/\qbar}, \zp(1)) &\mr{and}& 
	\breve{\mc{T}}_0 = H^1_{\et}(\tilde{Y}'_{/\qbar} , \tilde{C}_{0/\qbar} ,\zp(1)).
\end{eqnarray*}
Since $H^0_{\et}(\tilde{Y}'_{/\qbar},\tilde{C}_{0/\qbar},\zp(1)) = 0$,
the group $H^0_{\et}(\tilde{C}_{\infty/\qbar},\zp(1))$ provides the kernel in the exact sequence
\begin{equation} \label{lifted_seq}
	0 \to \La^{\iota}(1) \to \breve{\mc{T}}_{\infty,0}  \to \breve{\mc{T}}_0 \to 0
\end{equation}
of relative cohomology groups.
The pullback map $\breve{\mc{T}}_0 \to \tilde{\mc{T}}$ along the canonical section $Y_1(N) \to \tilde{Y}'$ has image the subgroup $\tilde{\mc{T}}'$ of
elements with constant residue at the infinity cusps and kernel $\bar{H}^0_{\et}(\tilde{C}_{0/\qbar},\zp(1)) \cong \bar{\La}(1)$.
We aim to modify \eqref{lifted_seq} in a similar manner to the construction of $\tilde{\mc{T}}^{\sharp}$ from $\tilde{\mc{T}}$.

The unique $I$-torsion lift $g = -t(\kappa)$ in $\tilde{\mc{T}}_{\Q}$ is $G_{\Q}$-fixed, and it
lies in $\tilde{\mc{T}}'_{\Q}$ since $\kappa$ is constant along the infinity cusps. In fact, $g$ comes from some
$g' \in H^1_{\et}(Y_1(N)_{/\Q},\qp(1))$, as the inflation map $H^2_{\et}(\Q,\qp(1)) \to H^2_{\et}(Y_1(N)_{/\Q},\qp(1))$
is split by pullback by a tangential basepoint $\vec{0}$ of $0 \in X_1(N)(\Q)$. We subtract the inflation of 
$\vec{0}^*(g') \in H^1_{\et}(\Q,\qp(1))$ from $g'$ 
so that $g'$ has trivial leading coefficient at $0$. Let $u_b \in H^1(\Q,\qp(1))$ denote the leading coefficient of $g'$ at the zero cusp 
$\langle b \rangle(0)$ for $b \in D$.
The quantity $a_{\ell} = (T_{\ell} -1 -\ell \langle \ell \rangle) g'$ is constant for each $\ell \nmid Mp$.
In particular, its leading coefficients are all equal, while the leading coefficient at $\langle b \rangle(0)$ is 
$(\ell-1)(u_b - u_{\ell^{-1}b})$. Summing this over all $b$ gives $|D|a_{\ell} = 0$, so $a_{\ell} = 0$. Thus, 
all $u_b$ are equal, hence trivial. The element $g'$ lifts uniquely to 
$H^1_{\et}(\tilde{Y}'_{/\Q} , \tilde{C}_{0/\Q} ,\qp(1))$, and we let $g_0$ denote its image in $(\breve{\mc{T}}_0)_{\Q}$.

As $(\La^{\iota}(1))^+ = 0$, there is a unique $\varsigma$-fixed lift $\tilde{g}$ of $g_0$ in $(\breve{\mc{T}}_{\infty,0})_{\Q}$. 
If $\lambda \in \La$ is such that $\lambda g_0 \in \breve{\mc{T}}_0$, then for any lift $v \in \breve{\mc{T}}_{\infty,0}$ of
$\lambda g_0$, we have $\lambda \tilde{g} = \frac{v+\varsigma v}{2} \in  \breve{\mc{T}}_{\infty,0}$ as well.
Setting $\breve{\mc{T}}_{\infty,0}^{\sharp} =  \breve{\mc{T}}_{\infty,0} + \La \tilde{g}$ and $\breve{\mc{T}}_0^{\sharp} = \breve{\mc{T}}_0 + \La g_0$, we therefore have an exact sequence
\begin{equation} \label{sharp_part_cpt_seq}
	0 \to \La^{\iota}(1) \to \breve{\mc{T}}_{\infty,0}^{\sharp}  \to \breve{\mc{T}}_0^{\sharp} \to 0
\end{equation}
of $\La$-modules.
\end{proof}

For a prime $\ell$, we denote level $N\ell$ objects with a superscript $(\ell)$: e.g., $X_1(N)^{(\ell)} = X_1(N\ell)$. 
We have degeneracy maps $p_1, p_{\ell} \colon X_1(N\ell) \to X_1(N)$ induced by multiplication by $1$ and $\ell$ on the upper half-plane, respectively. For $f \in \{p_1,p_{\ell}\}$ and $y \in C_0^{(\ell)} \cup C_{\infty}^{(\ell)}$,
we have $f^*t_{f(y)}= t_y^{e_y}$ locally at $y$ by our choices of Tate parameters, where $e_y$ is the ramification index of $y$. Consequently, $f$ defines a map $(\tilde{Y}')^{(\ell)} \to \tilde{Y}'$ with $f(\tilde{C}_0^{(\ell)}) = \tilde{C}_0$
and $f(\tilde{C}_{\infty}^{(\ell)}) = \tilde{C}_{\infty}$, and therefore $f^*$ defines morphisms $\breve{\mc{T}}_0 \to
\breve{\mc{T}}^{(\ell)}_0$ and $\breve{\mc{T}}_{\infty,0} \to \breve{\mc{T}}^{(\ell)}_{\infty,0}$.

Let $\mr{N}_{\ell} \colon \La \hookrightarrow \La^{(\ell)}$ be induced by the norm element for $\ker(D^{(\ell)} \to D)$.

\begin{step} Prove that the modified group $\breve{\mc{T}}_{\infty,0}^{\sharp}$ is $G_{\Q}$-stable. \end{step}

\begin{proof}[Sketch of proof.]
	For $\lambda \in J = (1-c^2[c] \mid (c,6Mp) = 1)$, we have $\lambda \kappa \in \mc{C}$. 
	As a consequence of Remark \ref{integral_Siegel} via the prior step, 
	the sharp construction using $t(\lambda \kappa)$ in place of $t(\kappa)$ 
	adds no new element to $\breve{\mc{T}}_{\infty,0}$. Thus, we may project 
	to the components of $\La$ on which $\La/J$ is supported, which is to say to the $\omega^{-2}$-component of the prime-to-$p$ part
	of $(\Z/Mp\Z)^{\times}$. All equalities below are implicitly taken in this localization.
	We are therefore done unless $p \mid N$ or $p = 3$.
	
	The $G_{\Q}$-stability of $\breve{\mc{T}}_{\infty,0}^{\sharp}$ is equivalent to the integrality of the coboundary 
	$w \colon G_{\Q} \to \La_{\Q}^{\iota}(1)$ of $\tilde{g}$, which we aim to prove. To show this, we employ auxiliary level
	structure, fixing a prime $\ell \nmid 6Mp$ such that $\ell \not\equiv 1 \bmod p$.
	Similarly to \cite[\S 6]{sst}, we take the level $N\ell$ Eisenstein
	series defined by 
	$$
		\mc{F}_{\ell}(q) = \ell \mc{E}(q^{\ell}) - \ell^{-1}[\ell]^{-1}\mc{E}(q)  \in \La_{\Q}\ps{q}.
	$$
	A computation gives $a_0(\mc{F}_{\ell})
	= - \frac{1}{2} \ell^{-1}[\ell]^{-1} (1-\ell^2[\ell])\xi^{\sharp}$, so $\mc{F}_{\ell}$ has $\La$-coefficients.
	
	Using $\mf{m}$ in this paragraph to denote localization at the complement of the union of the maximal ideals containing the
	Eisenstein ideal away from $\ell$, the $\mf{H}^{(\ell)}_{\mf{m}}$-module
	$\tilde{\mc{T}}^{(\ell)}_{\mf{m}}$ has a local quotient $\tilde{\mc{T}}_{\quo}^{(\ell)}$
	isomorphic to $\mf{M}_{\mf{m}}^{(\ell)}$. This canonically identifies the $[1]$-coefficient of $-\mc{F}_{\ell}$ with 
	\begin{equation} \label{modify_Siegel}
		h_{\ell} = p_1^*g - \ell^{-1}p_{\ell}^*([\ell]^{-1}g) \in ((\tilde{\mc{T}}^{(\ell)})_{\Q})^{G_{\Q}}.
	\end{equation}
	Let $\sigma$ in the inertia of $G_{\qp}$ restrict to $\varsigma$ on $\Q(\mu_{p^{\infty}})$.
	On the $\omega^{-2}$-component, $\sigma$ acts as $1$ on $\tilde{\mc{T}}_{\quo}^{(\ell)}$ and $-1$ 
	on $\ker(\tilde{\mc{T}}_{\mf{m}}^{(\ell)} \to \tilde{\mc{T}}_{\quo}^{(\ell)})$. Elements of $\tilde{\mc{T}}_{\quo}^{(\ell)}$ admit
	unique $\sigma$-fixed lifts to $\tilde{\mc{T}}_{\mf{m}}^{(\ell)}$, and the integral image of $h_{\ell}$ lifts to itself, 
	so $h_{\ell} \in \tilde{\mc{T}}^{(\ell)}$. We define $h_{\ell,0} \in (\breve{\mc{T}}_0^{(\ell)})_{\Q}$ as in \eqref{modify_Siegel} 
	but using $g_0$. It and its further lift to $\tilde{h}_{\ell} \in (\breve{\mc{T}}_{\infty,0}^{(\ell)})_{\Q}$ are the unique 
	$\varsigma$-fixed lifts of $h_{\ell}$, so integral.
	
	For $\sigma \in G_{\Q}$, one may compute that 
	$$
		\sigma \tilde{h}_{\ell} - \tilde{h}_{\ell} = \mr{N}_{\ell}(1-\ell^{-1}[\ell]^{-1})w(\sigma) \in (\La^{(\ell)})^{\iota}(1).
	$$
	As $\mr{N}_{\ell}$ maps $\La$ injectively onto a $\zp$-direct summand of $\La^{(\ell)}$, this implies 
	$(1-\ell^{-1}[\ell]^{-1})w(\sigma) \in \La^{\iota}(1)$.
	The quantity $1-\ell^{-1}[\ell]^{-1}$ is a unit in the $\omega^{-2}$-component of $\La$ in that it reduces to $1-\ell$,
	and we therefore have $w(\sigma) \in \La^{\iota}(1)$.
\end{proof}

\begin{step} Pull back by the lift of the residue to obtain the class. \end{step}

\begin{proof}[Details of construction.]
	Since \eqref{sharp_part_cpt_seq} is an exact sequence of $\La[G_{\Q}]$-modules,
	we may pull back \eqref{sharp_part_cpt_seq} via the $\La[G_{\Q}]$-module homomorphism 
	$\La \to \breve{\mc{T}}_0^{\sharp}$ sending $1$ to $-g_0$ to obtain an extension 
	\begin{equation} \label{cycl_unit_extn}
		0 \to \La^{\iota}(1) \to \mc{V} \to \La \to 0.
	\end{equation}
	The unit $z$ of interest to us is the value of $1$ under the connecting map 
	\begin{equation*} \label{conn_map_unit}
		\La = H^0_{\et}(\Z[\tfrac{1}{Mp}],\La) \to H^1_{\et}(\Z[\tfrac{1}{Mp}],\La^{\iota}(1)) \cong (\Z[\mu_N,\tfrac{1}{Mp}]^+)^{\times}
		 \otimes \zp.
	\end{equation*}
	That is, $z$ is the class of \eqref{cycl_unit_extn}.
\end{proof}

\begin{step} \label{Euler_relns} Read Euler system relations off of degeneracy maps on residues. \end{step}

\begin{proof}[Sketch of proof.]
	Fix a prime $\ell$, and set $n_{\ell} = \mr{N}_{\ell}(1) \in \La^{(\ell)}$. A calculation using the second Bernoulli distribution gives
	$n_{\ell} \kappa^{(\ell)} = \epsilon_{\ell} \kappa$,
	where $\epsilon_{\ell}$ denotes $p_{\ell}^*$ if $\ell \mid N$ and
	$p_{\ell}^* -  p_1^*\langle \ell \rangle$ if $\ell \nmid N$.
	The element $g \in \tilde{\mc{T}}'_{\Q}$ with divisor $-\kappa$ is uniquely determined as a $G_{\Q}$-fixed element, 
	so $n_{\ell} g^{(\ell)} = \epsilon_{\ell} g$. Moreover,
	$n_{\ell} g_0^{(\ell)} - \epsilon_{\ell} g_0 \in \bar{\La}_{\Q}^{(\ell)}(1)$ is $\varsigma$-fixed, hence zero. Similarly, 
	$n_{\ell} \tilde{g}^{(\ell)} = \epsilon_{\ell} \tilde{g}$ in that $n_{\ell} \tilde{g}^{(\ell)} - \epsilon_{\ell} \tilde{g}$ is $\varsigma$-fixed in 
	$(\La_{\Q}^{(\ell)})^{\iota}(1)$.

	Set $\rho_{\ell} = 1-[\ell]$ if $\ell \nmid N$ and $1$ otherwise. Since $n_{\ell} \tilde{g}^{(\ell)} = \epsilon_{\ell} \tilde{g}$, we have a map 
\begin{equation} \label{Euler_fact_map}
\begin{tikzcd}
	0 \arrow{r} & \La^{\iota}(1) \arrow{r} \arrow[d, "\mr{N}_{\ell}\rho_{\ell}"] & \mc{V} \arrow{r} \arrow[d] & \La \arrow{r} \arrow[d,"\mr{N}_{\ell}"] & 0 \\
	0 \arrow{r} & (\La^{(\ell)})^{\iota}(1) \arrow{r} & \mc{V}^{(\ell)} \arrow{r} & \La^{(\ell)}  \arrow{r} & 0
\end{tikzcd}
\end{equation}
of exact sequences
with $\mr{N}_{\ell}\rho_{\ell} $ induced by $\epsilon_{\ell} \colon \La^{\iota}(1) \to (\La^{(\ell)})^{\iota}(1)$. 
Corestriction on $p$-completions of unit groups arises from the projection map $q_{\ell} \colon (\La^{(\ell)})^{\iota}(1) \to \La^{\iota}(1)$ on coefficients, and $\mr{N}_{\ell}$ similarly induces restriction, which is injective. 

Let $\delta$ and $\delta^{(\ell)}$ denote the connecting maps on the zeroth cohomology of the sequences in \eqref{Euler_fact_map}. Since $\mr{N}_{\ell} q_{\ell}$ is multiplication by $n_{\ell} \in \La^{(\ell)}$, we see that
$$
	(\mr{N}_{\ell})_* (q_{\ell})_* \delta^{(\ell)} = \delta^{(\ell)}\mr{N}_\ell q_\ell = 
	(\mr{N}_{\ell})_* \rho_{\ell} \delta q_{\ell}.
$$ 
Since $z = \delta(1) = \delta(q_{\ell}(1))$, 
the norm $(q_{\ell})_*$ takes $z^{(\ell)} = \delta^{(\ell)}(1)$ to $\rho_{\ell} z$.
\end{proof}

The class $z_{\infty} = (z_r)_{r \ge 1}$ may be shown nontrivial in eigenspaces via a $p$-adic regulator computation. For completeness, but claiming no originality, we explain a modification of the method of the recent proof of Shang, Skinner, and Trejos Su\'arez \cite{sst} (cf. \cite{skinner,sask}) that applies to our construction. This uses semistable comparison for relative cohomology and an overconvergent primitive of the differential attached to the 
specialization of $\mc{E}$ at a character.

\begin{step} Verify character-wise nontriviality by a $p$-adic regulator computation. \end{step}

\begin{proof}[Sketch of proof.]
	Let $\theta$ be an even character of $(\Z/Mp\Z)^{\times}$.
	Let $\psi$ be a finite-order $p$-adic character of $D_{\infty}$ that agrees with $\theta$ on $(\Z/Mp\Z)^{\times}$ and
	is nontrivial on $1+p\zp$. We view $\psi$ as a primitive Dirichlet character of its conductor $f$ with $p \mid f$. We
	work until the last paragraph under the assumption that $N = f$.
	
	Let $K$ be a finite extension of $\qp(\mu_N, \psi(D))$
	over which $(X_1(N),C_1(N))$ attains semistable reduction. Let $e_{\psi} \in K[D]$ be the idempotent corresponding to $\psi$. 
	Let $V^{\psi} = e_{\psi}(\mc{V} \otimes_{\zp} K(-1))$, where $\mc{V}$ is as in \eqref{cycl_unit_extn}, 
	so that we have an exact sequence
	\begin{equation} \label{psi_seq}
		0 \to K \to V^{\psi} \to K(-1) \to 0
	\end{equation}
	of $K[G_K]$-modules. Suppose that \eqref{psi_seq} splits. 
	
	Consider the composition
	$$
		\beta \colon K(-1) \to V^{\psi} \hookrightarrow \mc{V} \otimes K(-1) \to \breve{\mc{T}}_{\infty,0} \otimes_{\zp} K(-1)
	$$
	determined by a splitting and the canonical map 
	$\mc{V} \to \breve{\mc{T}}_{\infty,0}^{\sharp}$.	
	Let $Z = \mc{O}_{Y_1(N)_{\infty}}(C_{\infty})^{\times}$, and let $k \colon Z \to Y_1(N)_{\infty}$ be the projection map.
	We pull back $\beta$ via the map $Z \to \tilde{Y}'$ given by setting the coordinates along the fibers of the zero cusps equal to the
	canonical sections $1$. Since $\psi$ is nontrivial while diamond operators act trivially on 
	the cokernel $H^0_{\et}(Y_1(N)_{\infty/\overline{K}},K(-1)) \cong K(-1)$ of $k^*$ on relative cohomology, the composition of $\beta$ with 
	the residue map to $K(-1)$ is trivial.
	Thus, the composition descends to a unique $K[G_K]$-module 
	homomorphism 
	$$
		\iota \colon K(-1) \to H^1_{\et}(Y_1(N)_{\infty/\overline{K}},C_{\infty/\overline{K}},K).
	$$
	
	Let $Q$ be the order of 
	the residue field of $K$, and let $\Phi$ be the $K$-linear Frobenius on $H^1_{\mr{dR}}(Y_1(N)_{\infty},C_{\infty}/K)$ obtained
	from the $Q$-power Frobenius on relative log-crystalline cohomology of the special fiber via semistable comparison.
	The image $v$ of a nonzero vector under the de Rham realization of $\iota$ lies in
	$\mr{Fil}^1(H^1_{\mr{dR}}(Y_1(N)_{\infty},C_{\infty}/K))$ and satisfies $\Phi(v) = Qv$
	(cf.~\cite[\S 2.2, \S 5.1]{sst}).
	
	Let $\tilde{\psi} \colon \La_{\Q}\ps{q} \to K\ps{q}$ be induced by $\psi$ on coefficients, and let $\tau(\psi)$ be the Gauss sum of $\psi$.
	The differential
	$$
		\omega_{\psi} = \frac{1}{\tau(\psi)}\tau_N^*\left(\tilde{\psi}(\mc{E})\frac{dq}{q}\right) = \sum_{n=1}^{\infty}\left( \sum_{d \mid n} \psi^{-1}
		\left(\frac{n}{d}\right)d \right) q^n \frac{dq}{q}
	$$
	has residue a $K^{\times}$-multiple of $e_{\psi} \kappa$. By its construction from $t(\kappa)$, 
	the image $v_{\mr{dR}}$ of $v$ in $H^1_{\mr{dR}}(Y_1(N)_{\infty}/K)$ has the same residues as $u \omega_{\psi}$ for some 
	$u \in K^{\times}$.
	Both $v_{\mr{dR}}$ and the class $[\omega_{\psi}]$ have $T_{\ell}$-eigenvalue $\ell+\psi^{-1}(\ell)$ at each prime $\ell \nmid N$
	for the usual Hecke action on differentials. 
	Ramanujan-Petersson therefore forces $v_{\mr{dR}} - u[\omega_{\psi}] \in H^1_{\mr{dR}}(X_1(N)/K)$ to be zero. As 
	$\mr{Fil}^1 H^0_{\mr{dR}}(C_{\infty}/K) = 0$, the class $v$ is represented by $(u\omega_{\psi},0)$.

	The following paragraph is adapted from the arguments of \cite[\S 5.2]{sst}.
	On the connected ordinary region $U$ of the rigid analytic curve $X_1(N)^{\mr{an}}_{/K}$ containing $C_{\infty}$,
	there is an overconvergent weight zero form $H_{\psi}$ with $dH_{\psi} = \omega_{\psi}$ having constant term
	$\frac{1}{2}\psi(a)L_p(\psi^{-1},1)$ at $\langle a \rangle^{-1}\infty$ for each $a \in D$.
	We restrict  $H^1_{\mr{dR}}(Y_1(N)_{\infty},C_{\infty}/K)$ to the first overconvergent cohomology group $\mc{H}$ for 
	$U - C'_{\infty}$ relative to $C_{\infty}$. Consider the exact sequence
	$$
		0 \to K \to H^0(C_{\infty},K) \xrightarrow{\partial^{\dagger}} \mc{H}.
	$$
	The image $v^{\dagger}$ of $v$ in $\mc{H}$ is $-u\partial^{\dagger}(H_{\psi}|_{C_{\infty}})$.
	As the cusps are $K$-rational, $\Phi(v^{\dagger}) = v^{\dagger}$. Since we also have 
	$\Phi(v^{\dagger}) = Qv^{\dagger}$ by comparison, we have $v^{\dagger} = 0$. Therefore $H_{\psi}|_{C_{\infty}}$ is constant,
	yet 
	$$
		H_{\psi}(\infty) - H_{\psi}(\langle a \rangle^{-1} \infty) = \frac{1-\psi(a)}{2}L_p(\psi^{-1},1)
	$$
	for any $a \in (\Z/N\Z)^{\times}$. 
	Choosing $a$ so that $\psi(a) \neq 1$, this forces $L_p(\psi^{-1},1) = 0$.
	
	No longer assuming $f = N$, choose $\psi$ of sufficiently large conductor 
	so that $L_p(\psi^{-1},1) \neq 0$ and $\psi(\ell) \neq 1$ for all $\ell \mid M$ with $\ell \nmid f$. 
	Write $z'_{\infty,\theta}$ for the analogue of $z_{\infty,\theta}$ with $M$ replaced by the prime-to-$p$ part of $f$.
	By Step \ref{Euler_relns}, the $\psi$-specialization of $\res_p z_{\infty,\theta}$ is the product of nonzero 
	Euler factors $\prod_{\ell \mid M, \ell \nmid f} (1-\psi(\ell))$ times the $\psi$-specialization of $\res_p z'_{\infty,\theta}$.
	The exact sequence \eqref{psi_seq} at level $f$ with class the $\psi$-specialization of $\res_p z'_{\infty,\theta}$
	does not split by the nonvanishing of $L_p(\psi^{-1},1)$, allowing us to conclude that 
	$z_{\infty,\theta} \neq 0$.
\end{proof}

Finally, we describe a spectral sequence argument from \cite{shwa} that shows the class $z$ of $\mc{V}$ is the real projection of $1-\zeta_N$ by employing the identification $g = g_{0,\frac{1}{N}}$ that we avoided above.

\begin{proof}[Sketch of argument.]
	We have an anticommutative square
	\begin{equation} \label{spec_seq_sq}
	\begin{tikzcd}
		H^1_{\et}(\tilde{Y}'_{/\Z[\frac{1}{Mp}]} ,\tilde{C}_{0/\Z[\frac{1}{Mp}]},\qp(1)) \arrow{r} \arrow{d} & 
		H^1_{\et}(\tilde{C}_{\infty/\Z[\frac{1}{Mp}]},\qp(1)) \arrow{d} \\
		H^1_{\et}(\tilde{Y}'_{/\qbar},\tilde{C}_{0/\qbar},\qp(1))^{G_{\Q}} \arrow[r,"\partial"] &
	 	H^1_{\et}(\Z[\frac{1}{Mp}],H^0_{\et}(\tilde{C}_{\infty/\qbar},\qp(1))).
	\end{tikzcd}
	\end{equation}
	Its vertical maps arise from Hochschild-Serre, the upper horizontal map is restriction to $\tilde{C}_{\infty}$, and $\partial$ is a 
	connecting map of \eqref{lifted_seq}. The $N$-unit $z$ is the value of $1$ under the connecting map induced by 
	\eqref{cycl_unit_extn}. That is, since $\mc{V}$ was constructed by pullback using $-g_0$, we have $z = -\partial(g_0)$.
	By the anticommutativity of \eqref{spec_seq_sq}, this is the leading coefficient  
	of $g'$ at $\infty$, i.e., the real projection $(1-\zeta_N)(1-\zeta_N^{-1}) \otimes \frac{1}{2}$ of $1-\zeta_N$.
\end{proof}

\subsection*{An application to our conjecture}

Theorem \ref{fkident} is the first part of a two-part theorem of Fukaya and Kato on Conjecture \ref{key_conj}. The second part \cite[Theorem 7.2.3(2)]{fk-sh} should extend to our more general setting. We state the result here under the original hypotheses \eqref{hypotheses}
and briefly outline the proof from \cite[\S 10.3]{fk-sh}.

\begin{theorem} 
	Let $\{0 \to \infty\}_{\theta} \in \tilde{\mc{T}}_{\mf{m},\theta}$ denote the image of the class of the path from $0$ to $\infty$, and consider
	the reduction $m_{\theta} \in P_{\theta}$ of $(U_p-1)\{0 \to \infty\}_{\theta} \in \mc{T}_{\mf{m},\theta}$. Then
	\begin{eqnarray*}
		\Upsilon_{\theta}(\{1-\zeta_N,p\}_{\theta}) =m_{\theta} &\mr{and}& \varpi_{\theta}(m_{\theta}) 
		= \{1-\zeta_N,p\}_{\theta},
	\end{eqnarray*}
	where the subscript $\theta$ on symbols denotes projection to the $\theta$-eigenspace.
\end{theorem}

\begin{proof}[Outline of proof.]
	For the second identity, see the calculation in \cite[\S 5]{sh-Lfn}. On the left of the first identity, we have 
	$\Upsilon_{\theta}$ applied to the value $\partial^1(p) = \{1-\zeta_N,p\}_{\theta}$ of the first connecting map $\partial^1$
	for the Tate twist of \eqref{cycl_unit_extn}.	
	On the right, we employ the global extension
	\begin{equation} \label{global_extn}
		0 \to T_{\theta} \to \tilde{\mc{T}}_{\mf{m},\theta}/I\mc{T}_{\theta} \to (\mf{H}/I)_{\theta} \to 0,
	\end{equation}
	with the image of $\{0 \to \infty\}_{\theta}$ mapping to $1 \in (\mf{H}/I)_{\theta}$. The
	pushout of this extension by $\phi_{\theta}$ agrees with the pushout by reduction modulo $\mc{L}_{\theta}$
	of the $\theta$-eigenspace of \eqref{cycl_unit_extn}, so we may compute $\partial^1(p)$
	using \eqref{global_extn}. The value $\Upsilon_{\theta} \circ \partial^1(p)$ 
	reduces to a local contribution at the primes dividing $Mp$. 
	
	The local contribution at $p$ is the value of $1$ under the composition
	\begin{equation} \label{local_contrib_p}
		H^0_{\et}(\Q,(\mf{H}/I)_{\theta}) \to H^1_{\et}(\Q,T_{\theta}) \to H^1_{\et}(\qp,P_{\theta}) 
		\xrightarrow{\cup p} H^2_{\et}(\qp,P_{\theta}(1)) \xrightarrow{\mr{inv}} P_{\theta}
	\end{equation}
	of the connecting map, the local splitting, cup product with $p$, and the invariant map.
	This is $\mr{inv}(\partial_p^0(1) \cup p)$ for the zeroth connecting map $\partial_p^0$ of the exact sequence of unramified
	local quotients of \eqref{global_extn}. The cocycle $\partial_p^0(1)$ is 
	determined by its value on $\Fr_p$. Since $\Fr_p$ acts as $U_p$ on $\tilde{\mc{T}}_{\quo}$, 
	this value is $m_{\theta}$.
	Rephrasing this, we have
	$\partial_p^0(1) = \nu  m_{\theta}$ for the unramified character
    	$\nu \colon G_{\qp} \to \zp$ with $\nu(\Fr_p) = 1$. Since
	$\mr{inv}(\nu \cup p) = 1$, the value of the composition is $m_{\theta}$. 
	
	Finally, for primes $\ell \mid M$, the primitivity of $\theta$
	gives a $G_{\Q_{\ell}}$-splitting of \eqref{keyseq}. Thus, we can write down the analogous composition to \eqref{local_contrib_p}
	using $\Q_{\ell}$-cohomology. As $\ell \not\equiv 1 \bmod p$, the $\mf{H}_{\theta}[G_{\Q_{\ell}}]$-extension of $(\mf{H}/I)_{\theta}$
	by $P_{\theta}$ is unramified at $\ell$.
	Since $p \in \Z_{\ell}^{\times}$, the composition then takes values in the zero unramified subgroup 
	$H^2(\F_{\ell},P_{\theta}(1))$ of $H^2(\Q_{\ell},P_{\theta}(1))$.
\end{proof}

\begin{remark}
We can unwind the first connecting map over $\Z[\tfrac{1}{Mp},\mu_N]^+$ of the Tate twist of \eqref{cycl_unit_extn} to a map
$$
	\Psi \colon H^1_{\et}(\Z[\tfrac{1}{Mp}],\La^{\iota}(1)) \otimes_{\zp} \La \to Y^{\circ} \otimes_{\zp} \La, \quad 
	\Psi( u \otimes 1) = \sum_{d \in D} \{ 1-\zeta_N^d, u \} [d].
$$
The inverse limit map $\Psi_{\infty}$ is valued in $Y_{\infty} \otimes_{\zp} \La_{\infty}$ and equals the composition of a reciprocity map \cite[\S 2]{sh-Lfn} with a map determined by the Galois action on $p$-power roots of cyclotomic $Mp$-units.
In $\theta$-eigenspaces under \eqref{hypotheses}, the composition $(\Upsilon_{\infty} \otimes 1) \circ \Psi_{\infty}$ recovers part of
a map which is conjectured in \cite[Conjecture 6.3]{sh-Lfn} to send limits of cyclotomic units to Eisenstein reductions of modified versions of the Mazur-Kitagawa $p$-adic $L$-function.\footnote{The full map can be recovered as a formal sum of such maps over divisors of $Mp$.} This conjecture is equivalent to the $\theta$-component of the inverse limit form of Conjecture \ref{key_conj}.
\end{remark}

\begin{ack}
	The author thanks Gyujin Oh and Preston Wake for our collaborations and the unfinished work described herein, Kazuya
	Kato, Takako Fukaya, and Akshay Venkatesh for their ideas in our earlier collaborations,
	Peter Xu for many fruitful conversations, Marco Sangiovanni Vincentelli and Chris Skinner for explaining their work on Euler systems,
	Sheng-Chi Shih and Jun Wang for explaining their work on $\Upsilon$, and Frederick Vu for his thesis work.
	
	The author employed versions of ChatGPT and Codex to test ideas, properly formulate and prove certain desired results, generate a
	 BibTeX file, proofread this article, and suggest wording. This work was supported in part by a Simons Foundation Travel Support for
	  Mathematicians award.
\end{ack}

\end{document}